\documentclass[11pt]{article}
\usepackage{tikz}
\usepackage{subfigure}
\usepackage[english]{babel}
\usepackage{graphicx}

\usepackage{amsfonts,amssymb,amsmath,latexsym,amsthm}
\usepackage{multirow}

\def\ex{\mbox{ex}}

\newtheorem{thm}{Theorem}[section]

\newtheorem{lem}[thm]{Lemma}

\newtheorem{prob}[thm]{Problem}
\newtheorem{cl}{\bf Claim}[section]

\begin{document}
\title{The Tur\'an number for the edge blow-up of trees}
\author{Anyao Wang$^a$,\ \  Xinmin Hou$^{a,b}$,\ \  Boyuan Liu$^a$,\ \  Yue Ma$^a$\\
\small$^a$School of Mathematical Sciences\\
\small University of Science and Technology of China, Hefei, Anhui, 230026, PR China\\
\small $^b$CAS Wu Wen-Tsun Key Laboratory of Mathematics\\
\small University of Science and Technology of China, Hefei, Anhui, 230026, PR China}

\maketitle

\begin{abstract}
The edge blow-up of a graph $F$ is the graph obtained from  replacing each edge in $F$ by a clique of the same size where the new vertices of the cliques are all different.
In this article, we concern about the Tur\'an problem for the edge blow-up of trees.
Erd\H{o}s et al. (1995) and Chen et al. (2003) solved the problem for stars. The problem for paths was resolved by Glebov (2011).
Liu (2013) extended the above results to cycles and a special family of trees with the minimum degree at most two in the smaller color class  (paths and  proper subdivisions of stars were included in the family).
In this article, we extend Liu's result to all the trees with the minimum degree at least two in the smaller color class. Combining with Liu's result, except one particular case, the Tur\'an problem for the edge blow-up of trees is completely resolved.  Moreover, we determine the maximum number of edges in
the family of $\{K_{1,k}, kK_2, 2K_{1,k-1}\}$-free graphs and the extremal graphs, which is an extension of a result given by Abbott et al. (1972).
\end{abstract}

\section{Introduction}
In this paper, all graphs considered are simple and finite. For a graph $G$ and a vertex $x\in V(G)$, the set of neighbors of $x$ in $G$ is denoted by $N_G(x)$, write $N_G[x]=N_G(x)\cup\{x\}$, called the closed set of neighbors of $x$. The {\it degree} of $x$, denoted by $\deg_G(x)$, is $|N_G(x)|$. Let  $\delta(G)$ and $\Delta(G)$  denote the minimum  and maximum degrees of $G$, respectively.
For a subset $X\subseteq V(G)$, let $\delta_{G}(X)=\min\{d_{G}(v) : v\in X\}$ and $\Delta_{G}(X)=\max\{d_{G}(v): v\in X\}$.
Let $e(G)$ be the number of edges of $G$. For a graph $G$ and $S, T\subset V(G)$, let $e_G(S, T)$ be the number of edges $e=xy\in{E(G)}$ with $x\in S$ and $y\in T$, if $S=T$, we use $e_G(S)$ instead of $e_G(S, S)$, and  $e_G(u, T)$ instead of $e_G(\{u\}, T)$ for convenience, the index $G$ will be omitted if no confusion from the context. For a subset $X\subseteq V(G)$ or $X\subseteq  E(G)$, let $G[X]$ be the subgraph of $G$ induced by $X$.
A {\it matching} $M$ in $G$ is a subset of $E(G)$ with $\delta(G[M])=\Delta(G[M])=1$. The {\it matching number} of $G$, denoted by $\nu(G)$, is the maximum number of edges in a matching in $G$.
For a given graph $G$, we write $\alpha(G)$ and $\beta(G)$ for the independent and  the (vertex) covering number of $G$, respectively. Write $\chi(G)$  for the chromatic number of $G$.  Given a family $\mathcal{L}$ of graphs, define $p(\mathcal{L})=\min_{L\in\mathcal{L}}\chi(L)-1$.

Given two graphs $G$ and $H$, we say that $G$ is {\it $H$-free} if $G$ does not contain an $H$ as a subgraph. The {\it Tur\'an number}, denoted by $\ex(n,H)$, is the largest number of edges of an $H$-free graph on $n$ vertices. That is,
$$\ex(n,H)=\max\{e(G):|V(G)|=n,  \mbox{ $G$ is $H$-free}\}.$$
We call an $H$-free graph with $n$ vertices and $\ex(n,H)$ edges an  {\it extremal graph} of $H$.
For positive integers $n$ and $r$ with $n\geq r$, the Tur\'an graph, denoted by $T_{n,r}$, is the complete $r$-partite graph on $n$ vertices with the size of each part differing by at most one.
Two fundamental theorems due to Mantel~\cite{M07} and Tur\'an~\cite{T41} state that $\ex(n, K_{r+1})=e(T_{n,r})$ for $r\ge 2$ and the Tur\'an graph is the unique extremal graph. The celebrated Erd\H{o}s-Stone-Simonovits Theorem~\cite{ES66,ES46} told us that asymptotically Tur\'an's construction is best possible for any graph $H$ with $\chi(H)\ge 3$.
But the Tur\'an problem are often very difficult for bipartite graphs even for the cycles and trees.  In this article, we mainly concern a special family of graphs obtained from  blowing up all edges in trees.   Formally, given a graph $F$ and a positive integer $p$, the {\it edge blow-up} of $F$, denoted by $F^{p+1}$, is  the  graph obtained from replacing each edge in $F$ by a clique of size $p+1$ where the new added vertices of the cliques are all different.

The extremal graphs of the edge blow-up of trees have special constructions. To describe the constructions, we need some definitions and notation.
Given two vertex-disjoint graphs $H_1$ and $H_2$, the {\it join graph} of $H_1$ and $H_2$, denote by $H_1\vee{H_2}$, is  the graph obtained by joining each vertex of $H_1$ to each vertex of $H_2$. Given vertex-disjoint graphs $H_1, \ldots, H_k$, we can define the join graph of $H_1, \ldots, H_k$ recursively, i.e. $H_1\vee \ldots\vee H_k=(H_1\vee \ldots\vee H_{k-1})\vee H_k$. Write $\vee_{i=1}^k H_i$ for $H_1\vee\ldots\vee H_k$ and $\cup_{i=1}^k H_i$ for the disjoint union of $H_1, \ldots, H_k$. If $H_1=\cdots=H_k=H$, write $kH$ for $\cup_{i=1}^k H_i$.
A graph is called {\it almost $d$-regular} if all vertices, except at most one of degree $d-1$, have degree $d$.
Write $R(n,d)$ for an almost $d$-regular graph on $n$ vertices.
Denote by $K(n_1, \ldots, n_p)$ the  complete $p$-partite graph with color classes of orders $n_1, \ldots, n_p$.
As usual, write $K_t$ and $K_{s,t}$ for $K(t)$  and $K(s,t)$ respectively. Write  $K_t^c$ for the empty graph on $t$ vertices.
Let $L_1(n_1, \ldots, n_p; k)$ be the family of graphs constructed by embedding a $R(2k-1, k-1)$ in one class of $K(n_1, \ldots, n_p)$ and
let $L_2(n_1, \ldots, n_p; k)$ be the family of graphs constructed by embedding  a $R(k+1,k-1)\cup K_{k-1}$ (for even $k$) or $2K_k$ (for odd $k$) in one class of $K(n_1, \ldots, n_p)$.
If $n_1, \ldots, n_p$ have almost equal size, i.e. $|n_i-n_j|\le 1$ for $1\le i,j\le p$, we write $K(n;p)$ and  $L_s(n; p;k)$ for $K(n_1, \ldots, n_p)$ and  $L_s(n_1,\ldots, n_p;k)$, respectively, where $n=\sum_{i=1}^p n_i$ and $s\in\{1,2\}$.

Now we define the extremal graphs for the blow-up of trees.
For positive integers $n, p, a, k$ with $n\ge a$, let
$$\mathcal{H}_i(n,p,a,k)=K_{a-1}\vee L_i(n_1,\ldots, n_p;k), \mbox{ for $i=1,2$},$$
and
$$\mathcal{H}_2(n,p,a,d,k)=R(a-1, d)\vee L_2(n_1,\ldots, n_p;k),$$
where $\sum\limits_{i=1}^pn_i=n-a+1$.
In particular, let
$${H}_i(n,p,a,k)=K_{a-1}\vee L_i(n-a+1;p;k), \mbox{ for $i=1,2$},$$
and
$${H}_2(n,p,a,d,k)=R(a-1, d)\vee L_2(n-a+1;p;k).$$
Note that $H_i(n,p,a,k)$ ($i=1,2$) and $H_2(n,p,a, d,k)$ are also families of graphs and by the definitions they might contain non-isomorphic graphs. However, as mentioned in~\cite{L13},   all members in each family have the same number of edges, and since their difference does not matter in this article, we will treat each of them as a "unique" graph instead of a family of graphs.

Define
$$g_1(k)=\left\{
                    \begin{array}{ll}
                    k^2-\frac 32k, & \hbox{$k$ is even;} \\
                    k^2-\frac{3k-1}2, & \hbox{$k$ is odd,}
                    \end{array}
                  \right.
\mbox{ and }
g_2(k)=\left\{
                    \begin{array}{ll}
                    k^2-\frac 32k, & \hbox{$k$ is even;} \\
                    k^2-k, & \hbox{$k$ is odd.}
                    \end{array}
                  \right.
$$
Let
$$g(n,p,a)=e(T_{n-a+1,p})+e(K_{a-1,n-a+1})+e(K_{a-1}),$$
and
$$g(n,p,a,b-1)=e(T_{n-a+1,p})+e(K_{a-1,n-a+1})+e(R(a-1,b-1)).$$

The Tur\'an problem for the edge blow-up of trees was originally studied by Erd\H{o}s et al.~\cite{Erdos-95}, they determined the value of $ex(n, F^{3})$ and the extremal graphs of $F^3$ when $F$ is a star. Here is a list of  some of known Tur\'an type results about the edge blow-up of trees.
\begin{itemize}
\item[(a).]
(Erd\H{o}s et al~\cite{Erdos-95}, 1995)
For $k\geq1$ and $n\geq 50k^2$, $$\ex(n, K_{1,k}^3)=g(n,2,1)+g_2(k).$$
Moreover, when $k$ is odd,  $H_2(n,2,1,k)$ is the unique extremal graph; when $k$ is even, $H_1(n,2,1,k)$ and $H_2(n,2,1,k)$  are extremal graphs.
\item[(b).]
(Chen et al.~\cite{Wei}, 2003) For any $p\ge 2$, $k\ge 1$ and  $n\ge 16k^3(p+1)^8$,
 $$\ex(n, K_{1,k}^{p+1})=g(n,p,1)+g_2(k),$$
where $H_2(n,p,1,k)$ is the unique extremal graph for odd $k$, and when $k$ is even, $H_1(n,p,1,k)$ and $H_2(n,p,1,k)$ are extremal graphs.

\item[(c).]
(Glebov~\cite{Glebov-2011}, 2011)
For any $p\ge 2$, $k\ge 1$ and $n>16r^{11}(p+1)^8$,
$$\ex(n, P_{r+1}^{p+1})=g(n,p,a)+g_2(k),$$
where $a=\left\lfloor\frac{r+1}2\right\rfloor$ and $k=1$ or $2$ with respect to $r$ is odd or even.
Moreover, $H_2(n,p,a,1)$ (resp. $H_1(n,p,a,2)$) is the unique extremal graph when $r$ is odd (when $r$ is even, resp.).

\item[(d)](Liu~\cite{L13}, 2013)\label{Liu13}
Given a tree $T$, denote by $A$ and $B$ its two color classes with $a=|A|\le |B|$.  For any $p\ge 3$, when $n$ is sufficiently large, we have that
\begin{enumerate}
  \item[(i)]
if $\delta(A)=1$ and $\alpha(T)=|B|$, then $\ex(n, T^{p+1})=g(n,p,a)+g_2(1)$. Moreover, $H_2(n,p,a,1)$ is the unique extremal graph for $T^{p+1}$.
  \item[(ii)]
if $\delta(A)=2$, then $\ex(n, T^{p+1})=g(n,p,a)+g_2(2)$. Moreover, $H_1(n,p,a, 2)$ is the unique extremal graph for $T^{p+1}$.
\end{enumerate}

\item[(e)]
There are some other related results: (i)  Liu~\cite{L13} also determined the Tur\'an number for the edge blow-up of cycles and its extremal graphs, which are almost the same as those for the blow-up of paths. (ii) Another interesting extension of the result (a) is blowing up every edge of a tree by a cycle of odd length instead of a clique. Hou et al~\cite{Ex for Ckq, Hou18} solved the problem for stars and recently Zhu et al~\cite{Zh20} resolved the problem for paths and cycles.
\end{itemize}

In this article, we extend Liu's result to trees $T$ such that the two color classes $A, B$ have the property that $|A|\le |B|$ and $\delta(A)\ge 2$.  The main result is as follows.

\begin{thm}\label{THM: Main}
Given $p\ge 3$ and a tree $T$ such that its two color classes $A$ and $B$ satisfying $|A|\le{|B|}$.  Let $A_0=\{ x\in A : d(x)=\delta(A)\}$ and $B_0=\{ y\in B : |N(y)\cap A_0|\ge 2\}$. Denote by $a=|A|$, $k=\delta(A)$ and $b+2=\delta(B_0)$.  If $k\ge 2$ then,  when $n$ is sufficiently large, we have
$$\ex(n, T^{p+1})=\left\{
                    \begin{array}{ll}
                     g(n,p,a)+g_1(k), & \hbox{$k$ is even;}\\
                     g(n,p,a)+g_2(k), & \hbox{$k$ is odd and $B_0=\emptyset$;} \\
                     g(n,p,a)+g_1(k), & \hbox{$k$ is odd and $b=0$ or $0<b\le a-1-\lceil{\frac{k-1}{a-1}\rceil}$;}\\
                     g(n,p,a,b-1)+g_2(k), & \hbox{$k$ is odd and $b\ge \max\{1,a-1-\lceil{\frac{k-1}{a-1}\rceil}\}$.}
                    \end{array}
                  \right.
$$
Furthermore, for even $k$, ${H}_1(n,p,a,k)$ and ${H}_2(n,p,a,k)$ are extremal graphs; for odd $k$, if $B_0=\emptyset$ then ${H}_2(n,p,a,k)$ is the unique extremal graph,
if $B_0\not=\emptyset$ and $b=0$ or $0<b<a-1-\lceil{\frac{k-1}{a-1}\rceil}$ then ${H}_1(n,p,a,k)$ is the unique extremal graph, if $B_0\not=\emptyset$ and  $b>\max\{0, a-1-\lceil{\frac{k-1}{a-1}\rceil}\}$ then ${H}_2(n,p,a,b-1,k)$ is the unique extremal graph, and if $B_0\not=\emptyset$ and $b=a-1-\lceil{\frac{k-1}{a-1}\rceil}>0$ then ${H}_1(n,p,a,k)$ and ${H}_2(n,p,a,b-1,k)$ are extremal graphs.
\end{thm}

\noindent{\bf Remarks:} 
(1) For even $k$, $H_2(n,p,a,k)$ is a new extremal graph which is not mentioned  in~\cite{Wei,Erdos-95}.

(2) Combining the result given by Liu~\cite{L13},
the Tur\'an problem for the edge blow-up of trees $T$, except for the case $\delta(A)=1$ and $\alpha(T)\not=|B|$, is resolved.

The results of another type of extremal problems will be used in the proof of Theorem~\ref{THM: Main}.
Abbott, Hanson, and Sauer~\cite{AHS72} determined the maximum number of edges in a graph with maximum degree and matching number no more than $k$.
\begin{thm}[\cite{AHS72}]\label{LEM: AHS}
Let $G$ be a $\{K_{1,k}, kK_2\}$-free graph. Then
$$e(G)\le g_2(k)=\begin{cases}
	k^2-k & \text{if } k \text{ is odd},\\
	k^2-\frac{3}{2}k & \text{if } k \text{ is even}.
	\end{cases}$$
Moreover, the equality holds if $G=2K_k$ when $k$ is odd, and $G=R(2k-1-2t, k-1)\cup tK_{1,k-1}$ for $0\le t\le \frac{k}{2}-1$ or $R(k+1, k-1)\cup K_{k-1}$ when $k$ is even.

\end{thm}

We further determine the maximum number of edges in $\{K_{1,k}, kK_2, 2K_{1,k-1}\}$-free graphs and its extremal graphs in the following theorem, which has its own flavor in extremal graph theory and will be used in the proof of Theorem~\ref{THM: Main}.
\begin{thm}\label{LEM: 2K_{1,k-1}}
Let $G$ be a $\{K_{1,k}, kK_2, 2K_{1,k-1}\}$-free graph.		
Then $$e(G)\le g_1(k)=\begin{cases}
	k^2-\frac{3}{2}k+\frac{1}{2} & \text{if } k \text{ is odd},\\
	k^2-\frac{3}{2}k & \text{if } k \text{ is even}.
	\end{cases}$$
Furthermore, the equality holds if and only if $G=R(2k-1, k-1)$ when $k$ is odd, and $G=R(2k-1, k-1)$ or $R(k+1, k-1)\cup K_{k-1}$ when $k$ is even. If we do not care the the trivial components, the optimal graphs are determined completely.
\end{thm}

Some other (including hypergraph) extensions of Theorem~\ref{LEM: AHS} have been made e.g., in~\cite{BK09,CH76,HYGL19,K14}, especially the following theorem given by Chv\'atal and Hanson~\cite{CH76} will be used in our proof.

\begin{thm}[\cite{CH76}]\label{LEM: CH}
	For all graph $G$ with maximum degree $\Delta\ge{1}$ and matching number $\nu\ge{1}$, then $e(G)\le{f(\nu,\Delta)}=\nu\Delta+\lfloor{\frac{\Delta}{2}}\rfloor\left\lfloor{\frac{\nu}{\lceil{\Delta/2}\rceil}}\right\rfloor$.
\end{thm}

The rest of the article is arranged as follows. We give preliminaries in the next section and prove Theorem~\ref{LEM: 2K_{1,k-1}} in Section 3. The proof of Theorem~\ref{THM: Main} will be given in Section 4. We give some discussions in the last section.
\section{Preliminaries}

A graph $G$ on $n$ vertices is called an $n$-vertex graph. Write $H\subseteq G$ for $H$ being a subgraph of $G$.
The following family of graphs was first introduced by Simonovits~\cite{M74}.

\newtheorem{def1}{\bf Definition}[section]
\begin{def1}
Let $\mathcal{D}(n,p,r)$ be the family of $n$-vertex graphs $G$ satisfying the following symmetric conditions:
	\begin{enumerate}
		\item[(1)] It is possible to omit at most $r$ vertices of $G$ such that the remaining graph $G{'}$ is a join of graphs of almost equal order, i.e. $G{'}=\vee_{i=1}^{p}G^i$ where $|V(G^i)|=n_i$ and $|n_i-n/p|\le{r}$ for any $i\in [p]$.
The vertices in $V(G)\setminus V(G{'})$ are called the {\it exceptional vertices}.

		\item[(2)] For every $i\in[p]$, there exit connected graphs $H_i$ such that $G^i={k_i}H_i$ where $k_i=n_i/{|H_i|}$ and any two copies $H_i^j$, $H_i^\ell$ in $G^i$ ($1\le{j}<{\ell}\le{k_i}$) are symmetric subgraphs of $G^i$, i.e. there exists an isomorphism $\omega: H_i^j\longrightarrow{H_i^\ell}$ such that for every $x\in{V(H_i^j)}$, $u\in{V(G^i)\setminus(V(H_i^j\cup H_i^\ell))}$, $xu\in{E(G^i)}$ if and only if $\omega(x)u\in{E(G^i)}$.
The graphs $H_i$ are called the blocks.

	\end{enumerate}

\end{def1}

\noindent{\bf Remark:} For given integers $p,a,k$ and large enough $n$,  $H_1(n,p,a,k)$, $H_2(n,p,a,k)$ and $H_2(n,p,a,b-1,k)$ are graphs in the family $\mathcal{D}(n,p,r)$ for appropriate value of $r$.

The following two theorems also due to Simonovits~\cite{M74,M99} are the base of the proof of Theorem~\ref{THM: Main}.

\begin{thm}[\cite{M74}]\label{thm2}
	Assume that a finite family $\mathcal{L}$ of forbidden graphs with $p(\mathcal{L})=p$ is given. If for some $L\in{\mathcal{L}}$ and $\ell=|V(L)|$,$$L\subseteq{P_\ell\vee{K((p-1)\ell; p-1)}}$$ then there exist $r=r(L)$ and $n_0=n_0(r)$ such that $\mathcal{D}(n,p,r)$ contains an $\mathcal{L}$-extremal graph for every $n>n_0$. Furthermore, if this is the only extremal graph in $\mathcal{D}(n,p,r)$, then it is the unique extremal graph for every sufficiently large $n$.
\end{thm}

\begin{thm}[\cite{M99}]\label{thm3}
	Assume that a finite family $\mathcal{L}$ of forbidden graphs with $p(\mathcal{L})=p$ is given. If for some $L\in{\mathcal{L}}$ and $\ell=|V(L)|$, $$L\subseteq{\ell P_2\vee{K(2(p-1)\ell; p-1)}}$$ then there exit $r=r(L)$ and $n_0=n_0(r)$ such that $\mathcal{D}(n,p,r)$ contains an $\mathcal{L}$-extremal graph for every $n>n_0$. Furthermore, for any $\mathcal{L}$-extremal graph $G\in{\mathcal{D}(n,p,r)}$, we have that
	\begin{enumerate}
		\item[(i)] all blocks of $G$ consist of isolated vertices, i.e. the join graph $G'$ will be a Tur\'{a}n graph $T_{n', p}$.\label{itref}
		\item[(ii)] each exceptional vertex in $V(G)\setminus V(G')$ is joined either to all the vertices of $G{'}$ or to all the vertices of $p-1$ classes of $G{'}$ and to no vertex of the remaining class.\label{itref}
	\end{enumerate}
\end{thm}

In the following lemma, we show that $T^{p+1}$ satisfies the condition of Theorems~\ref{thm2} and~\ref{thm3}.
\begin{lem}\label{LEM: T^{p+1}}
For $p\ge 3$ and $\ell=|V(T^{p+1})|$,
$$T^{p+1}\subseteq P_\ell\vee K((p-1)\ell; p-1) \mbox { and } T^{p+1}\subseteq \ell P_2\vee K(2(p-1)\ell; p-1).$$
\end{lem}
\begin{proof}
Let $S_1, \ldots, S_{p-1}$ be the $p-1$ classes of $K((p-1)\ell; p-1)$. Then $|S_1|=\ldots=|S_{p-1}|=\ell$. Note that $\nu(P_\ell)=\lfloor\frac \ell2\rfloor$. Hence $P_\ell$ has a maximum matching $M$ of size $\lfloor\frac \ell2\rfloor$.
Since $|S_1|=|S_2|>|V(T)|$, we can embed $T$ into $S_1\vee S_2$. Label the edges of $T$ by $e_1, \ldots, e_m$. Now we blow up the edges of $T$ as follows: for $e_i\in E(T)$, choose $v^i_3\in S_3, \ldots, v^{i}_{p-1}\in S_{p-1}$ and an edge $v^i_pv^i_{p+1}$ in $M$. Then $V(e_i)\cup\{v^i_3, \ldots, v^i_{p+1}\}$ induces a copy of $K_{p+1}$ in $P_\ell\vee K((p-1)\ell; p-1)$. Note that $|S_i|>|M_i|=\lfloor\frac \ell2\rfloor\ge m$. This guarantees  that the blow-ups are vertex disjoint.

Clearly, $\ell P_2$ is a matching of size $\ell$ and the size of each class of $K(2(p-1)\ell; p-1)$ is bigger than the one of $K((p-1)\ell; p-1)$. So we can embed $T^{p+1}$ into $K(2(p-1)\ell; p-1)$ in the same way.

\end{proof}

The following definitions and lemma due to Liu~\cite{L13} also play an important role in our proof.

\begin{def1}\label{vertex split}
(1)	Given a graph $H$, a vertex split on a  vertex $v\in{V(H)}$ is defined as follows:
	replace $v$ by an independent set of size $d(v)$ in which each vertex is adjacent to exactly one distinct vertices in $N_H(v)$. Given a vertex subset $U\subseteq{V(H)}$, a vertex split on $U$ means applying vertex split on the vertices in $U$ one by one. The splitting family of $H$, denoted by $\mathcal{H}(H)$, is  the family of graphs that can be obtained from $H$ by applying vertex split on some $U\subseteq{V(H)}$.

(2) Given a family $\mathcal{L}$, let $\mathcal{M}:=\mathcal{M}(\mathcal{L})$ be the family of minimal graphs $M$ that satisfy the following:
there exits an $L\in{\mathcal{L}}$ and a $t=t(L)$ such that $L\subseteq{M'\vee{K((p-1)t; p-1)}}$, where $M'=M\cup{K_t^c}$ (putting $M$ into an independent set of size $t$). We call $\mathcal{M}$ the decomposition family of $\mathcal{L}$.
\end{def1}

\begin{lem}[\cite{L13}]\label{LEM: M=H}
Given $p\ge{3}$ and any graph $H$ with $\chi(H)\le{p-1}$, $\mathcal{M}(H^{p+1})=\mathcal{H}(H)$. In particular, a matching of size $e(H)$ is in $\mathcal{M}(H^{p+1})$.
\end{lem}

The following lemma shows that the split operations increase the matching number.

\begin{lem}\label{LEM: split}
Given a graph $H$ and let $\mathcal{H}(H)$ be its splitting family. Then $\nu(H')\ge \nu(H)$ for any $H'\in\mathcal{H}(H)$.
\end{lem}
\begin{proof}
Let $H'\in\mathcal{H}(H)$ and assume $H'$ is obtained from $H$ by applying vertex split on $U\subseteq V(H)$. It is sufficient to show $\nu(H')\ge \nu(H)$ when $|U|=1$. Assume $U=\{u\}$.  Let $M$ be a maximum matching in $H$. If $u$ is not covered by $M$ then $M$ is still a matching in $H'$. We are done. Now assume $u$ is covered by $M$ and the edge $uv\in M$. By the definition of vertex split, there is a copy $u'$ of $u$ with $u'v\in E(H')$. Then the edge set obtained from replacing $uv$ in $M$ by $u'v$ is a matching in $H'$,  we are done.
\end{proof}



We also need the following fundamental results in graph theory.
\begin{thm}[\cite{BM76}]
\begin{itemize}
\item[(1)] (Hall's Theorem)\label{THM: Hall}
 A bipartite graph $G$ with bipartition $(X,Y)$ has a matching which covers every vertex in $X$ if and only if $|S|\le |N(S)|$ for all $S\subseteq X$.

\item[(2)] (K\"{o}nig's Theorem) If $G$ is a bipartite graph, then $\beta(G)=\nu(G)$.

\item[(3)] (Gallai's Theorem) If $G$ is a graph, then $\alpha(G)+\beta(G)=|V(G)|$.	
\end{itemize}
\end{thm}
We will prove some lemmas for preparation.
\begin{lem}\label{LEM: Beta-(i)-(iii)}
Let $T$ be a tree with two color classes $A$ and $B$. Assume $|A|\le{|B|}$ and $\delta(A)\ge{2}$. Then the following holds.
\begin{enumerate}
\item[(i)] $\beta(T)=|A|$;

\item[(ii)] $\min\{\beta(T') : T'\in\mathcal{H}(T)\}\ge |A|$;

\item[(iii)] $K_{|A|-1}\vee {K_m^c}$ is $\mathcal{H}(T)$-free for any $m\ge{1}$.

\end{enumerate}
\end{lem}

\begin{proof}
(i) We first claim that $\nu(T)=|A|$. For any $S\subseteq{A}$,
$\delta(S)\ge \delta(A)\ge2$. So $2|S|\le e_T(S, N(S))\le |S|+|N(S)|-1$. Thus $|S|\le{|N(S)|}$ for any $S\subseteq A$.
By Hall's Theorem, $T$ has a matching saturated $A$.
That is, $\nu(T)=|A|$. By K\"onig's Theorem, $\beta(T)=\nu(T)=|A|$.

(ii) By Lemma~\ref{LEM: split}, $\nu(T')\ge \nu(T)=|A|$ for all $T'\in\mathcal{H}(T)$. By K\"onig's Theorem, $\beta(T')=\nu(T')\ge |A|$, we are done.

(iii) Clearly, $\beta(K_{|A|-1}\vee K_m^c)\le |A|-1$. The statement follows directly from (ii).
\end{proof}

In what follows, given a tree $T$, let $A, B, A_0, B_0$ and $a, b, k, p$ be defined the same as in Theorem~\ref{THM: Main}.
Note that in the rest of this section we always assume $k\ge 2$.
Given a $\tilde{T}\in\mathcal{H}(T)$, let $S$ be the set of splitting vertices in $A$ and let $\tilde{S}$ be the set of vertices split from $S$. Denote $s=|S|$ and $m=|\tilde{S}|$.
Let $\tilde{A}=(A\setminus S)\cup \tilde{S}$ and $\tilde{B}=V(\tilde{T})\setminus{\tilde{A}}$.

\begin{lem}\label{LEM: splitting on $A$}
If $s\ge 1$ then	$\nu(\tilde{T})\ge a-1+k$.
\end{lem}	

\begin{proof}
	By Lemma~\ref{LEM: split}, we need only to prove the lemma for $s=1$. Then $m\ge k$. Let $S=\{v\}$.
	Since $|\tilde{A}|\ge a-1+m\ge a-1+k$, by Hall's Theorem, it suffices to prove that for any $A'\subseteq \tilde{A}$, $|N_{\tilde{T}}(A')|\ge|A'|$.
	
	Fix an $A'\subseteq \tilde{A}$, let $A'_1=A'\cap \tilde{S}$ and $A'_2=A'\cap (A\backslash\{v\})$.
	Note that for any $u\in A'_2$, we have that $d_{\tilde{T}}(u)\ge 2$ and $|N_{\tilde{T}}(u)\cap N_{\tilde{T}}(v)|\le 1$ (since $\tilde{T}$ is a forest).
	Thus $$|N_{\tilde{T}}(A')|=|N_{\tilde{T}}(A'_1)|+|N_{\tilde{T}}(A'_2)\backslash N_{\tilde{T}}(A'_1)|\ge |A'_1|+|A'_2|=|A'|.$$
\end{proof}

\begin{lem}\label{LEM: EX-k-even}

(1)If $k\equiv0\pmod2$ then all graphs in $\mathcal{H}_1(n,p,a,k)\cup \mathcal{H}_2(n,p,a,k)$ are $T^{p+1}$-free.

(2) If $k\equiv 1\pmod2$ and $B_0=\emptyset$ then all graphs in $\mathcal{H}_2(n,p,a,k)$ are $T^{p+1}$-free.

(3)
If $k\equiv 1\pmod2$, $B_0\not=\emptyset$ and $b=0$ or $0<b<a-1-\lceil{\frac{k-1}{a-1}\rceil}$, then all graphs in $\mathcal{H}_1(n,p,a,k)$ are $T^{p+1}$-free.

(4)
If $k\equiv 1\pmod2$, $B_0\not=\emptyset$ and $b>\max\{0, a-1-\lceil{\frac{k-1}{a-1}\rceil}\}$, then all graphs in $\mathcal{H}_2(n,p,a,b-1,k)$ are $T^{p+1}$-free.

(5) If $k\equiv 1\pmod2$, $B_0\not=\emptyset$ and $b=a-1-\lceil{\frac{k-1}{a-1}\rceil}>0$, then all graphs in $\mathcal{H}_1(n,p,a,k)\cup\mathcal{H}_2(n,p,a,b-1,k)$ are $T^{p+1}$-free.

\end{lem}

\begin{proof}
	Choose $G\in \mathcal{H}_1(n,p,a,k)\cup \mathcal{H}_2(n,p,a,k)\cup \mathcal{H}_2(n,p,a,b-1,k)$. Without loss of generality, assume that $L_i(n_1,\ldots, n_p; k)$ $(i=1,2)$ is obtained by embedding a member, say $G_0$, of $\{R(2k-1, k-1), R(k+1, k-1)\cup K_{k-1}, 2K_k\}$ in the first class, denoted by $X_1$,  of $K(n_1,\ldots, n_p; p)$.
	Let $C_1=G[X_1]$. Then $\nu(C_1)\le k-1$ and $\Delta(C_1)\le k-1$.
	By Lemma~\ref{LEM: M=H}, it suffices to prove that $K_{a-1}\vee{C_1}$ or $R(a-1, b-1)\vee C_1$ is $\mathcal{H}(T)$-free.
	Write $K$ for $K_{a-1}$ or $R(a-1,b-1)$.
	
	Suppose to the contrary that there is a $\tilde{T}\in\mathcal{H}(T)$ such that $\tilde{T}\subseteq K\vee{C_1}$.
	Denote $A{'}=\tilde{A}\cap V(C_1)$, $B'=\tilde{B}\cap V(K)$ and $B''=\tilde{B}\setminus B'$. Let $t'=|B'|$.
	
	Let $A_1'$ and $B_1''$ be the sets of non-isolated vertices of $\tilde{T}[A', B'']$ in $A'$ and $B''$, respectively.
	Let $T'=\tilde{T}[A'_1, B''_1]$ be the forest consisting of nontrivial components of $\tilde{T}[A', B'']$.
	Then $\delta(T')\ge 1$.
	If $s\ge 1$, then by Lemma~\ref{LEM: splitting on $A$}, we have that $\nu(\tilde{T})\ge a-1+k$.
	Note that $B'$ is a vertex cover of $T-V(T')$.  By the K\"{o}nig's Theorem, $\nu(T-V(T'))=\beta(T-V(T'))\le |B'|\le a-1$. Thus $\nu(T')\ge \nu(T)-\nu(T-V(T'))\ge k$, a contradiction to $\nu(T')\le \nu(C_1)\le k-1$.
	
	Now assume $s=0$. Then $|A'|\ge t'+1$. Since $\tilde{T}[A{'}, B{'}]$ is a forest, we have  $e_{\tilde{T}}[A', B']\le |A'|+|B'|-1$.
	Therefore,
	\begin{eqnarray*}
		e(T')\ge k|A'|-(|A'|+|B'|-1)\ge (k-2)|A'|+2\ge (k-2)(t'+1)+2.
	\end{eqnarray*}

If $t'=0$ there will be at least one vertex in $A'_1$ with $d_{T'}(v)\ge k$, a contradiction to $\Delta(T')\le\Delta(C_1)\le k-1$.
	
If $t'= 1$, by the Pigeonhole principle, either there exits one vertex $u_0\in A_1'$ with $d_{T'}(u_0)\ge{k}$ or there are two vertices $u_1, u_2\in A_1'$ with $d_{T'}(u_i)= k-1$ for $i=1,2$.
For the former, we can get a contradiction with the same reason as the case $t'=0$.  For the latter, we have $N_T(u_1)\cap N_T(u_2)=B'$.
So $T'$ has two disjoint copies of $K_{1,k-1}$. By $|V(G_0)|\ge |V(T')|\ge 2k$. We have $T'=2K_{1,k-1}$, $G_0=2K_k$, and $G\in\mathcal{H}_2(n,p,a,b-1,k)$. And $T'=2K_{1,k-1}$ also implies that $A'=A_1'=\{u_1,u_2\}$ (otherwise, any vertex of $A'\setminus A_1'$ has degree at most $|B'|=1$, a contradiction to $\delta(A)\ge 2$). By the assumption of (4), $b>0$ and the vertex of $B'$ has degree at most $b-1$ in $R(a-1, b-1)$. So the degree of the vertex of $B'$ has degree at most $b+1$, a contradiction to $\delta(B_0)=b+2$.
	
	Now assume $t'\ge 2$. Then
\begin{equation}\label{EQ: V(T')}
2k=|V(G_0)|\ge |V(T')|\ge e(T')+1\ge (k-2)(t'+1)+3\ge 3k-3.
\end{equation}
When $k\ge 4$, we have $2k=|V(G_0)|\ge|V(T')|\ge 2k+1$ , a contradiction.
When $k=3$, we have that all the equalities hold in (\ref{EQ: V(T')}).  Hence $G_0=2K_k$ and $e(T')=5$.
But this is impossible since a spanning forest of $G_0$ has at most four edges.
	When $k=2$, we have $G_0\cong R(3,1)$ or $R(3,1)\cup K_1$ by the definition of $\mathcal{H}_i(n,p,a,k)$ for $i=1,2$. So $e(G_0)=1$. But $e(T')\ge 2$, a contradiction to $T'\subseteq G_0$.
\end{proof}

\section{Maximal $\{K_{1,k}, kK_2, 2K_{1,k-1}\}$-free graphs}

\begin{proof}
Suppose $G$ is a $\{K_{1,k}, kK_2, 2K_{1,k-1}\}$-free graph with maximum number of edges.
By Theorem~\ref{LEM: AHS}, $e(G)\le g_2(k)$. When $k$ is even, $e(G)\le k^2-\frac{3}{2}k$ and the equality holds if $G=R(2k-1, k-1)$ or $R(k+1, k-1)\cup K_{k-1}$. Clearly, $R(2k-1, k-1)$ and $R(k+1, k-1)\cup K_{k-1}$ are $2K_{1,k-1}$-free, so it is sufficient to show that there is no other extremal graph regardless of trivial components when $k$ is even.

First we note that $\nu(G)={k-1}$ (otherwise $G\cup{K_{1,k-2}}$ has more edges  but it is still $\{K_{1,k}, kK_2, 2K_{1,k-1}\}$-free, a contradiction).
We shall make use of the Gallai-Edmonds structure theorem~\cite{AK11}:
$G$ has a subset $S\subseteq V(G)$ such that (i) $$\nu(G)=\frac{1}{2}(|V(G)|+|S|-o(G-S)),$$ where $o(G-S)$ is the number of odd components of $G-S$; (ii) every odd component of $G-S$ is factor-critical; (iii) every even component of $G-S$ has a perfect matching; (iv) every maximum matching $M$ in $G$ saturates $S$, and every edge of $M$ incident with $S$ joins a vertex in $S$ to a vertex in an odd component of $G-S$.
Denote by $|S|=s$ and $o(G-S)=m$. Then $n+s-m=2k-2$.
Let $G_1, G_2,\ldots, G_{m'}$ be all the components of $G-S$. Then $m\le m'$. Let $n_i=|V(G_i)|$ for $i=1,2,\ldots, n_{m'}$. Without loss of generality, assume $n_1, \ldots, n_m$ are odd with $n_1\ge\ldots\ge{n_{m}}$ and $n_{m+1}, \ldots, n_{m'}$ are even with $n_{m+1}\ge \ldots\ge n_{m'}$.
Define $$f(x)=
	\begin{cases}
		{x\choose 2}	 & \text{if } x\le{k-1},\\
	\lfloor{\frac{(k-1)x}{2}}\rfloor & \text{if } x\ge{k}.
	\end{cases}$$
\begin{cl}\label{CL:f(x)}
(a) $f(x)$ is strictly increasing in $[1,+\infty)$.

(b) $f(x_1)+f(x_2)\le f(x_1+x_2-1)$ for odd integers $1\le x_1\le k-1$ and $x_1\le x_2$. Moreover, the equality holds if and only if $x_1=1$ when $x_1+x_2\le k$, and $x_1=1$ or $k-1$ when $x_2\ge k$.
\end{cl}
\begin{proof}[Proof of Claim~\ref{CL:f(x)}:]
(a) It can be checked directly from the definition of $f(x)$.

(b) If $1\le x_1\le x_2\le k-1$ and $x_1+x_2\le k$, then
\begin{eqnarray*}
&&f(x_1+x_2-1)-f(x_1)-f(x_2)\\
&&=\frac{(x_1+x_2-1)(x_1+x_2-2)}2-\frac{x_1(x_1-1)}2-\frac{x_2(x_2-1)}2\\
                          &&=(x_1-1)(x_2-1)\\
                          &&\ge  0,
\end{eqnarray*}
and the equality holds if and only if $x_1=1$.

If $1\le x_1\le x_2\le k-1$ and $x_1+x_2\ge k+1$, then $k+1\le x_1+x_2\le 2k-2$. Hence we have $k\ge 3$.
So
\begin{eqnarray*}
f(x_1+x_2-1)-f(x_1)-f(x_2)
&=&\left\lfloor\frac{(k-1)(x_1+x_2-1)}2\right\rfloor-\frac{x_1(x_1-1)}2-\frac{x_2(x_2-1)}2\\
                          &\ge&\frac{(k-x_1)x_1+(k-1-x_2)(x_2-1)-1}2\\
                          &\ge & \frac{k-2}2\\
                          &> &0.
\end{eqnarray*}

If $1\le x_1\le k-1<x_2$, note that $x_1$ and $x_2$ are odd, then
\begin{eqnarray*}
f(x_1+x_2-1)-f(x_1)-f(x_2)
&=&\left\lfloor\frac{(k-1)(x_1+x_2-1)}2\right\rfloor-\frac{x_1(x_1-1)}2-\left\lfloor\frac{(k-1)x_2}2\right\rfloor\\
                          &=&\frac{(k-1-x_1)(x_1-1)}2\\
                          &\ge & 0.
\end{eqnarray*}
So we have $f(x_1+x_2-1)\ge f(x_1)+f(x_2)$ and the equality holds if and only if $x_1=1$ or $k-1$.

\end{proof}

\noindent{\bf Case 1:} $s=0$.

Then $n-m=2k-2$.  

\begin{cl}\label{CL:m=m'}
		$m=m'$.
\end{cl}
Suppose to the contrary that $m<m'$.
Note that $$n_i+n_j\le{n-(m'-2)}\le{n-(m+1-2)}=2k-1$$ for any $1\le{i}\neq{j}\le{m'}$.
This implies that there is at most one component with at least $k$ vertices.
By Claim~\ref{CL:f(x)}, $$e(G_1)+e(G_{m+1})\le f(n_1)+f(n_{m+1})< f(n_1+n_{m+1}).$$
Let $G'$ be the graph obtained from $G$ by deleting the two components $G_1$ and $G_{m+1}$ and adding a new $n_1+n_{m+1}$-vertex component $G_1'$ with $\Delta(G_1')\le k-1$ and $f(n_1+n_{m+1})$ edges.
Then $\nu(G')\le\frac{|V(G')|+|S|-o(G'-S)}{2}=k-1$. Since $\Delta(G_1')\le k-1$, we have $\Delta(G')\le k-1$ too.  Since $n_1+n_{m+1}\le 2k-1$, $G_1'$ is $2K_{1,k-1}$-free. Note that $G_1$ and $G_{m+1}$ are largest odd and even components in $G$, respectively, and $G$ has at most one component of order at least $k$. So $G'$ is also $2K_{1,k-1}$-free. But $e(G')>e(G)$, a contradiction to the maximality of $e(G)$.

To complete the case, it is sufficient to show the following claim.
\begin{cl}
Either $n_1=2k-1$ and $n_j=1$ for $2\le{j}\le{m}$ or $n_1=k+1$, $n_2=k-1$ and $n_j=1$ for $3\le{j}\le{m}$ when $k$ is even.
\end{cl}
Recall that $n_1\ge n_2\ge \ldots\ge n_m$ and $n_i+n_j\le n-(m-2)=2k$.	So $n_1\le 2k-1$ and $1\le n_i\le k$.
Suppose $n_2\not=1$ and $k-1$.
	By Claim~\ref{CL:f(x)},  $f(n_1)+f(n_2)<f(n_1+n_2-1)$.
	Since $G$ is $2K_{1,k-1}$-free, there is at most one component $G_i$ with $\Delta(G_i)=k-1$. Without loss of generality, assume $\max\{\Delta(G_1),\Delta(G_2)\}=k-1$ if any.
	Let $G'$ be the graph obtained from $G$ by replacing components $G_1, G_2$ by two new components $G_1', G_2'$ such that $G_1'\cong R(n_1+n_2-1, k-1)$ and $G_2'\cong K_1$. Then $e(G_1')=f(n_1+n_2-1)$. Note that $n_1+n_2-1\le 2k-1$ and $\Delta(G_i)<k-1$ for $3\le i\le m$. It is easy to check that $G'$ is $\{K_{1,k}, kK_2, 2K_{1,k-1}\}$-free. But this is a contradiction to the maximality of $e(G)$.
	So $n_2=1$ or $k-1$. If $n_2=1$ then $n_2=\ldots=n_m=1$ because of $1\le n_m\le\ldots\le n_2=1$  and $n_1=n-(m-1)=2k-1$.
	If $n_2=k-1$ ($k$ must be even in this case), then by Claim~\ref{CL:f(x)}, $n_1\ge k$ is odd and $n_1+n_2\le 2k$, which implies that $n_1=k+1$ and $n_j=1$ for $3\le{j}\le{m}$.
 This completes the proof of the case $s=0$.

\medskip
\noindent{\bf Case 2:} $s\ge{1}$.
\medskip

By the Gallai-Edmonds structure theorem, $G_1\cup\ldots\cup G_{m'}$ is $\{K_{1,k}, (k-s)K_{2}, 2K_{1,k-1}\}$-free. So, from  Case 1,  we have $\sum_{i=1}^{m'}e(G_i)\le f(2k-2s-1)$, and the equality holds if and only if $m=m'$, $n_1=2k-2s-1$ and $n_i=1$ for $2\le i\le m$, {or $n_1=k-s+1$,  $n_2=k-s-1$ and $n_i=1$ for $3\le i\le m$ when $k-s$ is even (in fact, the case holds only for $s=1$)}. Now we claim that
$e(G)<f(2k-1)$ in this case. Note that
\begin{equation}\label{EQ: e3}
e(G)=e(S)+e[S,\cup_{i=1}^{m'}V(G_i)]+\sum_{i=1}^{m'}e(G_i)\le(k-1)s+f(2k-2s-1)\le f(2k-1),
\end{equation}
and the equality holds if and only if $d_G(v)=k-1$ for each $v\in S$, $m=m'$, $n_1=2k-2s-1\ge k$ and $n_i=1$ for $2\le i\le m$, or $n_1=k-s+1$,  $n_2=k-s-1$ and $n_i=1$ for $3\le i\le m$ when $k-s$ is even.
Therefore, the equality holds in (\ref{EQ: e3}) implies that $G_1=R(2k-2s-1, k-1)$ or $G_1=R(k-s+1, k-1)$ when $k-s$ is even and $d_G(v)=k-1$ for each $v\in S$.
If $k$ is odd, since $G$ is $K_{1,k}$-free,  any vertex in $S$ can not be adjacent to vertices in $V(G_1)$. So a vertex in $S$  and its neighbors form a $K_{1,k-1}$, which is disjoint with a $K_{1,k-1}$ in $G_1$, a contradiction.
 Now suppose $k$ is even. Since $n_1\ge k$ is odd, we have $n_1\not=k-s+1$, i. e., $n_1=2k-2s+1\ge k+1$. Let $u\in V(G_1)$ be the vertex with degree $k-2$. Since $G$ is $K_{1,k}$-free,  any vertex in $S$ can not be adjacent to vertices in $V(G_1)\backslash \{u\}$. So a vertex in $S$  and its neighbors form a $K_{1,k-1}$, which is disjoint with a $K_{1,k-1}$ in $G_1\backslash \{u\}$, a contradiction.

The proof of the theorem is completed.
\end{proof}

\section{Proof of Theorem~\ref{THM: Main}}
Define
$$h(n,p,a,k)=\left\{
                    \begin{array}{ll}
                     g(n,p,a)+g_1(k), & \hbox{$k$ is even;}\\
                     g(n,p,a)+g_2(k), & \hbox{$k$ is odd and $B_0=\emptyset$;} \\
                     g(n,p,a)+g_1(k), & \hbox{$k$ is odd and $b=0$ or $0<b\le a-1-\lceil{\frac{k-1}{a-1}\rceil}$;}\\
                     g(n,p,a,b-1)+g_2(k), & \hbox{$k$ is odd and $b\ge \max\{1,a-1-\lceil{\frac{k-1}{a-1}\rceil}\}$.}
                    \end{array}
                  \right.
$$
By Lemma~\ref{LEM: T^{p+1}}, $T^{p+1}\subseteq \ell P_2\vee K(2(p-1)\ell; p-1)$ and $T^{p+1}\subseteq P_\ell\vee K((p-1)\ell; p-1)$. By Theorems~\ref{thm2} and~\ref{thm3}, there exit $r=r(T^{p+1})$ and $n_0=n_0(r)$ such that $\mathcal{D}(n,p,r)$ contains a $T^{p+1}$-extremal graph  for every $n>n_0$. Furthermore,  every $T^{p+1}$-extremal graph $G\in \mathcal{D}(n,p,r)$ satisfies (i) and (ii) in Theorem~\ref{thm3}. Choose such  a $T^{p+1}$-extremal graph $G$ with maximum number of edges. By Lemma~\ref{LEM: EX-k-even},
we have $e(G)\ge h(n,p,a,k)$. Denote by $W$ the set of vertices in $V(G)\setminus V(G')$ which is joined to all the vertices of $G'$. Denote by $B_i$ the set of vertices in $V(G)\setminus V(G')$ which is joined to no vertices of $G^i$ but to all the vertices of the remaining $p-1$ classes. Let $C_i=G[V(G^i)\cup{B_i}]$.

\begin{cl}
	$|W|=a-1$.
\end{cl}

\begin{proof}
	First, we show that	$|W|\le{a-1}$. Otherwise, $T$ can be embedded into $G[W]\vee{ C_1}$ since $|W|\ge |A|$ and $|V(C_1)|\ge |B|$, but this is impossible since $T\in\mathcal{H}(T)=\mathcal{M}(T^{p+1})$ by Lemma~\ref{LEM: M=H}.
	Now we claim that $|W|=a-1$. If not, suppose $|W|=a-1-s$ for some $s\ge{1}$. Then
	$$\begin{aligned}
	e(G)&\le e(T_{n-a+1+s,p})+(a-1-s)n+o(n)\\
	&=\frac{p-1}{2p}n^2+\frac{a-1-s}{p}n+o(n)\\
	&<\frac{p-1}{2p}n^2+\frac{a-1}{p}n+o(n)\\
	&=h(n,p,a,k),
	\end{aligned}$$
a contradiction to the maximality of $e(G)$.
\end{proof}


\begin{cl}\label{CL:Delta+nu}
	For any $1\le i\le p$ and any vertex $u\in{C_i}$, $d_{C_i}(u)+\sum\limits_{j\neq{i}}\nu(C_j)\le{k-1}$.
\end{cl}

\begin{proof}
	Suppose to the contrary that there exits an $i$ and  a vertex $u_0\in{V(C_i)}$  such that $d_{C_i}(u_0)+\sum\limits_{j\neq{i}}\nu(C_j)\ge{k}$. Without loss of generality, assume $u_0\in{C_1}$ and $N_{C_1}(u_0)=\{u_1,\ldots, u_s\}$. Let $x_1y_1, \ldots, x_ty_t$ be a matching in $\cup_{j=2}^p C_j$  with $s+t=k$. Choose $\alpha\in A$ with minimum degree in $T$ and let $T'=T-\alpha$. Then $d_T(\alpha)=k$ and
$T'$ has color classes $A\setminus\{\alpha\}$ and $B$. Note that $|A\setminus\{\alpha\}|=a-1=|W|$ and $n$ is large enough.
So we can embed $T'$ into $G$ such that $A\setminus\{\alpha\}=W$ and $B\subseteq (\cup_{i=1}^p V(C_i))\setminus\{u_0, u_1, \ldots, u_s, x_1, y_1, \ldots, x_t, y_t\}$.
Therefore, $T$ can be embedded into $G$ by putting $\alpha$ at $u_0$ and its neighbors at $u_1, \ldots, u_s, x_1, \ldots, x_t$.
Now we show that $T^{p+1}$ can be embedded in $G$ too.
For each edge  $\beta_i\gamma_j\in E(T')$ with $\beta_i\in{W}$ and $\gamma_j\in V(C_q)$ for some $1\le{q}\le{p}$.  Choose exactly one vertex, denoted by $\theta_{ij}^\ell$, in each $V(G^\ell)$ for $1\le{\ell}\le{p}$ and $\ell\neq{q}$. Then $G[\{\beta_i,\gamma_j\}\cup \{\theta_{ij}^\ell : 1\le{\ell}\le{p}   \mbox{ and } \ell\neq{q}\}]\cong K_{p+1}$ is a blow-up of  the edge $\beta_i\gamma_j\in E(T)$.
For each $u_0u_i\in E(T)$, $1\le{i}\le{s}$, choose exactly one vertex, denoted by $w_i^\ell$, in each $V(C_\ell)$ for $2\le{\ell}\le{p}$. Then $G[\{u_0,u_i,w_i^2,\ldots, w_i^p\}]\cong K_{p+1}$ forms a blow-up of the edge $u_0u_i$ for $1\le i\le s$.
For each $u_0x_i\in E(T)$ ($1\le{i}\le{t}$), assume $x_i\in{C_q}$ for some $2\le{q}\le{p}$. Choose exactly one vertex, denoted by $z_i^\ell$ in each $V(C_\ell)$ for $2\le{\ell}\le{p}$ and $\ell\neq{q}$. Then $G[\{u_0,x_i,y_i\}\cup \{z_i^\ell : 2\le{\ell}\le{p} \mbox{ and } \ell\neq{q}\}]\cong K_{p+1}$ is a blow-up of  the edge $u_0x_i$ for $1\le i\le t$. Since $n$ is sufficiently large and $|V(C_i)|\sim\frac {n}p$, we can choose the sets $\{\theta_{ij}^\ell : 1\le{\ell}\le{p}   \mbox{ and } \ell\neq{q}\}$, $\{w_i^2,\ldots, w_i^p\}$ and $\{z_i^\ell : 2\le{\ell}\le{p} \mbox{ and } \ell\neq{q}\}$ are pairwise disjoint, a contradiction to $G$ is $T^{p+1}$-free.
\end{proof}

\begin{cl}
	$\max\limits_{1\le i\le p}\Delta(C_i)=k-1$.
\end{cl}
\begin{proof}
If not then $\Delta(C_i)\le{k-2}$ for any $1\le{i}\le{p}$. To get a contradiction, it suffices to prove that $\sum_{i=1}^pe(C_i)<k^2-\frac{3}{2}k$.
If $k=2$, then $\Delta(C_i)=0$ for any $1\le{i}\le{p}$, we are done. Now assume $k\ge 3$.
	
If there exits an $i$ such that $\nu(C_i)\le\Delta(C_i)$, then by Theorem~\ref{LEM: CH} and Claim~\ref{CL:Delta+nu},
$$\begin{aligned}
	\sum_{j=1}^pe(C_j)&\le{\sum_{j=1}^pf(\nu(C_j),\Delta(C_j))}\\
	&\le{\sum_{j=1}^p{\nu(C_j)(\Delta(C_j)+1)}}\\
	&\le{(k-1)\left(\sum_{j\neq{i}}\nu(C_j)+\nu(C_i)\right)}\\
	&\le{(k-1)\left(k-1-\Delta(C_i)+\nu(C_i)\right)}\\
	&\le{(k-1)^2}\\
    &< k^2-\frac{3}2k.
	\end{aligned}$$

If $\nu(C_i)\ge{\Delta(C_i)+1}$ for any $1\le{i}\le{p}$ then $\Delta(C_i)\ge 1$. Denote by $s=\min\limits_j\{\nu(C_j)-\Delta(C_j)\}$. Then $\nu(C_j)\ge{s+1}$ for any $1\le{j}\le{p}$. By Claim~\ref{CL:Delta+nu}, we have $\Delta(C_j)\le{k-1-(p-1)(s+1)}$ for any $1\le{j}\le{p}$. Choose $i$ such that $\nu(C_i)-\Delta(C_i)=s$.  Again by Theorem~\ref{LEM: CH} and Claim~\ref{CL:Delta+nu}, we have
$$\begin{aligned}
	\sum_{j=1}^pe(C_j)&\le{\sum_{j=1}^pf(\nu(C_j),\Delta(C_j))}\\
    &\le{\sum_{j=1}^p{\nu(C_j)(\Delta(C_j)+1)}}\\
	&\le{\left(\sum_{j\neq{i}}\nu(C_j)+\nu(C_i)\right)[k-(p-1)(s+1)]}\\
	&\le{(k-1+s)[k-(p-1)(s+1)]}\\
	&=k^2-(ps-2s+p)k-(s-1)(p-1)(s+1)\\
	&<{k^2-\frac32k}.
	\end{aligned}$$
This completes the proof of the claim.
\end{proof}

Without loss of generality, assume that $\Delta(C_1)=k-1$. By Claim~\ref{CL:Delta+nu}, we have $\sum\limits_{j=2}^p\nu(C_j)=0$, i.e.
$e(C_j)=0$ \mbox{ for any $2\le j\le p$}, and $\nu(C_1)\le k-1$. So
\begin{equation}\label{EQ: e4}
\sum_{i=1}^p e(C_i)=e(C_1)\le f(\Delta(C_1), \nu(C_1))\le f(k-1,k-1)=g_2(k).
\end{equation}
Therefore,
\begin{eqnarray}\label{EQ: e5}
\begin{aligned}
	e(G)&\le e(C_1\vee\ldots\vee C_p)+e(W, \cup_{i=1}^p V(C_i))+e(G[W])+\sum_{i=1}^p e(C_i)\\
    &\le e(T_{n-a+1,p})+(a-1)(n-a+1)+e(G[W])+e(C_1)\\
    &\le g(n,p,a)+g_2(k),
	\end{aligned}
\end{eqnarray}
and the equality holds if and only if $C_1\vee\ldots\vee C_p\cong T_{n-a+1,p}$, $G[\cup_{i=1}^pV(C_i)\cup W]\cong K_{a-1}\vee T_{n-a+1,p}$, $G[W]\cong K_{a-1}$,
and $e(C_1)=g_2(k)$.
So we have $e(G)\le h(n,p,a,k)$ when $k$ is even or $k$ is odd and $B_0=\emptyset$. Note that $e(G)\ge h(n,p,a,k)$ by assumption and $g_1(k)=g_2(k)$ when $k$ is even. Therefore, $e(G)=h(n,p,a,k)$ and
$G\cong H_1(n,p,a,k)$ or $H_2(n,p,a,k)$ for even $k$ and $G\cong H_2(n,p,a,k)$ when $k$ is odd and $B_0=\emptyset$.

Now suppose $k$ is odd and $B_0\not=\emptyset$.	We have the following claim.
\begin{cl}\label{CL:G[W]}
Either $C_1$ is $2K_{1,k-1}$-free or $G[W]$ is $K_{1,b}$-free.
\end{cl}
\begin{proof}
Assume that there exist two disjoint copies of $K_{1,k-1}$ in $C_1$ and a copy of $K_{1,b}$ in $G[W]$.
We claim that $T$ can be embedded in  $G[W]\vee C_1$ and so $T^{p+1}\subseteq G$ by Lemma~\ref{LEM: M=H}, this is a contradiction.
Choose a vertex $u_0\in{B_0}$ with $d_T(u_0)=b+2$ and its two neighbors $u_1, u_2$ in $A_0$. Then $d_T(u_1)=d_T(u_2)=k$ and $(N_T[u_1]\setminus\{u_0\})\cap (N_T[u_2]\setminus\{u_0\})=\emptyset$.
Then we can embed $u_0$ at the center of a copy of $K_{1,b}$ in $G[W]$ and embed $N_T[u_1]\setminus\{u_0\}$ and $N_T[u_2]\setminus\{u_0\}$ into two disjoint copies of $K_{1,k-1}$ in $C_1$ (this can be done by embedding $u_1$ and $u_2$ at the centers of the two $K_{1,k-1}$).
Next, embedding $A\setminus\{u_1, u_2\}$ into $W\setminus\{u_0\}$ (this can be done since $|A\setminus\{u_1,u_2\}|=|W\setminus\{u_0\}|=a-2$) and the rest vertices of  $B\setminus\{u_0\}$ into the vertices not used before in $C_1$ (it can be done since $|V(C_1)|$ is big enough than $|B|$), we get an embedding of $T$ in $G[W]\vee C_1$. The claims holds.

	
	
	
	
\end{proof}

By Claim~\ref{CL:G[W]} and Theorem~\ref{LEM: 2K_{1,k-1}}, we have $e(C_1)\le g_1(k)$ or $G[W]\cong R(a-1,b-1)$ ($b>0$).
By (\ref{EQ: e5}),
$$e(G)\le\max\{g(n,p,a)+g_1(k), g(n,p,a,b-1)+g_2(k)\},$$
and the equality holds if and only if
\begin{itemize}
\item[(i)]
 $G\cong H_1(n,p,a,k)$ when $b=0$ or $b>0$ and $g(n,p,a)+g_1(k)>g(n,p,a,b-1)+g_2(k)$ which is equivalent to $0<b< a-1-\lceil{\frac{k-1}{a-1}\rceil}$;
\item[(ii)]
$G\cong H_2(n,p,a,b-1)$ when $b>0$ and $g(n,p,a)+g_1(k)<g(n,p,a,b-1)+g_2(k)$, which is equivalent to $b> \max\{0,a-1-\lceil{\frac{k-1}{a-1}\rceil}\}$;

\item[(iii)]
$G\cong H_1(n,p,a,k)$ or $H_2(n,p,a,b-1)$ when $b>0$ and $g(n,p,a)+g_1(k)=g(n,p,a,b-1)+g_2(k)$, which is equivalent to $b=a-1-\lceil{\frac{k-1}{a-1}\rceil}>0$.

\end{itemize}
The uniqueness of the extremal graphs for the cases (i) and (ii) comes from Theorem~\ref{thm2} directly.

This completes the proof of Theorem~\ref{THM: Main}.


\section{Discussions and remarks}
Combining Theorem~\ref{THM: Main} and the result (d) given by Liu~\cite{L13}, the Tur\'an problem for $T^{p+1}$ with $p\ge 3$ has been resolved except for the case $\delta(A)=1$ and $\alpha(T)\not=B$. So there are two natural problems to be considered further.
\begin{prob}
(P1) Given $p\ge 3$ and a tree $T$ such that its two color classes $A$ and $B$ satisfying $|A|\le{|B|}$, determine $\ex(n,T^{p+1})$ when $\delta(A)=1$ and $\alpha(T)\not= B$.

(P2) Determine $\ex(n,T^{p+1})$ when $p=2$.
\end{prob}

\noindent{\bf Remarks:}
(1) For (P1), let $\mathcal{C}_T=\{ C : C\subseteq T \mbox{ and $V(C)$ is a vertex cover of $T$}\}$ and $K$ be a $\mathcal{C}_T$-free graph on $a-1$ vertices. We guess that the graph with  the structure $K\vee T_{n-a+1,p}$ is an extremal graph for $T^{p+1}$. The difficult thing is that we have no idea about the structure of $K$ and its number of edges.

(2) For (P2), we need new techniques to avoid Lemma~\ref{LEM: M=H} (given by Liu~\cite{L13}, which requires $p\ge 3$), and the performance for $p=2$ maybe different from $p\ge 3$.





\begin{thebibliography}{99}
\bibitem{AHS72}
 H. L. Abbott, D. Hanson, H. Sauer, Intersection theorems for systems of sets, J. Combin. Theory Ser. A 12 (1972), 381-389.

\bibitem{AK11}
J. Akiyama, M. Kano, Factors and factorizations of graphs: Proof techniques in factor theory,
Springer-Verlag Berlin Heidelberg, (2011).

\bibitem{BK09}
 N. Balachandran, N. Khare, Graphs with restricted valency and matching number, Discrete Math. 309(12)(2009), 4176--4180.


\bibitem{BM76}
 J. A. Bondy and U. S. R. Murty, Graph Theory with Applications, The Macmillan Press Ltd., London and Basingstoke (1976).

\bibitem {Wei}
	G. Chen, R. J. Gould, F. Pfender, and B. Wei, Extremal graphs for intersecting cliques, {\it J. Combin. Theory Ser. B} {\bf 89} (2003) 159--171.

\bibitem{CH76}
V. Chv\'atal, D. Hanson, Degrees and matchings, J. Combin. Theory Ser. B, 20(2) (1976), 128-138.

\bibitem{Erdos-95}
P. Erd\H{o}s, Z. F\"uredi, R. J. Gould, D. S. Gunderson, Extremal graphs for intersecting triangles, {\it J. Combin. Theory Ser. B} {\bf 64(1)} (1995) 89--100.

\bibitem{ES66}
P. Erd\H{o}s and M. Simonovits. A limit theorem in graph theory. Studia Sci. Math. Hungar., 1 (1966): 51-57.
\bibitem{ES46}
P. Erd\H{o}s and A. Stone. On the structure of linear graphs. Bull. Amer. Math. Soc., 52 (1946): 1087-1091.

\bibitem{Glebov-2011}
	R. Glebov. Extremal graphs for clique-paths. arXiv:1111.7029v1, 2011.





\bibitem{Ex for Ckq}
X. Hou, Y. Qiu and B. Liu, Extremal graph for intersecting odd cycles, {\it Electron. J. Combin.}, 23(2) (2016), P2.29.

\bibitem{Hou18}
X. Hou, Y. Qiu, B.  Liu, Tur\'an number and decomposition number of intersecting odd cycles. Discrete Math. 341 (2018), 126-137.




\bibitem{HYGL19}
X. Hou, L. Yu, J. Gao, B. Liu, 	The size of 3-uniform hypergraphs with given matching number and codegree,   Discrete Math. 342(2019), 760--767.

	
\bibitem{L13}
H. Liu, Extremal graphs for blow-ups of cycles and trees, Electron. J. Combin.,  20(1) (2013), P65.

\bibitem{K14}
 N. Khare, On the size of 3-uniform linear hypergraphs, Discrete Math. 334(2014), 26--37.

\bibitem{M07}
W. Mantel, Problem 28, Wiskundige Opgaoen, 10 (1907), 60-61.

\bibitem{M74}
 M. Simonovits, Extremal graph problems with symmetrical extremal graphs, additional chromatic conditions, Discrete Math.,  7 (1974), 349-376.

\bibitem{M99}
M. Simonovits. How to solve a Tur\'an type extremal graph problem? DIMACS Series in Discrete Mathematics and Theoretical Computer Science, 49, Amer. Math. Soc., Providence, RI, 1999.

\bibitem{T41}
 P. Tur\'an, On an extremal problem in graph theory (in Hungarian), Mat. Fiz. Lapok, 48(1941), 436-452.

\bibitem{Zh20}
H. Zhu,  L. Kang, E. Shan, Extremal graphs for odd-ballooning of paths and cycles, Graphs and Combinatorics, 36 (2020): 755-766.




\end{thebibliography}
\end{document}